\documentclass[11pt]{amsart}
\usepackage{amsmath,amsfonts,amssymb,amsthm,epsfig,color}
\newcommand{\bb}{\mathbb}

\newcommand{\C}{\bb C}
\newcommand{\h}{\bb H}
\newcommand{\Z}{\bb Z}
\newcommand{\R}{\bb R}
\newcommand{\N}{\bb N}
\newcommand{\Q}{\bb Q}

\newcommand{\s}{\bb S}
\newcommand{\semidirect}{\ltimes}

\newcommand{\q}{\mathcal Q}
\newcommand{\B}{\mathcal B}

\newcommand{\sgn}{\operatorname{sgn}}
\newcommand{\vv}{\mathbf{v}}

\newcommand{\hh}{\mathcal H}
\newcommand{\Exp}{\operatorname{Exp}}

\newcommand{\minuszero}{\backslash\{0\}}
\newcommand{\La}{\Lambda}
\newtheorem{Theorem}{Theorem}
\numberwithin{Theorem}{section}
\newtheorem{Cor}[Theorem]{Corollary}
\newtheorem{Claim}[Theorem]{Claim}
\newtheorem{Prop}[Theorem]{Proposition}
\newtheorem{lemma}[Theorem]{Lemma}

\newtheorem*{remark*}{Remarks}
\newtheorem*{lemma*}{Lemma}

\newtheorem*{claim}{Claim}
\newtheorem*{theorem*}{Theorem}

\numberwithin{equation}{section}
\begin{document}
\title{Cusp Excursions on Parameter Spaces}
\author{Jayadev S.~Athreya}

\subjclass[2000]{primary: 37A17; secondary 37-06, 37-02}
\email{jathreya@iilinois.edu}
\address{Deptartment of Mathematics, University of Illinois Urbana-Champaign, 1409 W. Green Street, Urbana, IL 61801}
\begin{abstract}
We prove several results for dynamics of $SL(d, \R)$-actions on non-compact parameter spaces by studying associated discrete sets in Euclidean spaces. This allows us to give elementary proofs of logarithm laws for horocycle flows on hyperbolic surfaces and moduli spaces of flat surfaces. We also give applications to quantitative equidistribution and Diophantine approximation.
\end{abstract}
\maketitle

\section{Introduction}\label{sec:intro}

The homogeneous space $SL(2, \R)/SL(2, \Z)$ can be viewed from many different perspectives. It is the unit tangent bundle of a non-compact, finite-volume hyperbolic orbifold; it is the moduli space of holomorphic quadratic differentials on tori, and is the space of unimodular lattices in $\R^2$ (see \S\ref{sec:ident} for details on these identifications). Each of these interpretations provide a different family of generalizations:

\medskip

\begin{description}

\item[Hyperbolic Surfaces] $SL(2, \R)/\Gamma$, where $\Gamma$ is a non-uniform lattice in $SL(2, \R)$. \medskip

\item[Quadratic Differentials] $\q_g$, the moduli space of quadratic differentials on compact genus $g$ Riemann surfaces.\medskip

\item[Unimodular Lattices] $X_d: = SL(d, \R)/SL(d, \Z)$, the space of unimodular lattices in $\R^d$, $d \geq 2$.\medskip

\end{description}

\medskip

A common thread is that they are non-compact spaces with a continuous action of a Lie group $G$, and a $G$-equivariant association of a discrete set of vectors in a Euclidean space to each point in the parameter space. The actions of these groups, and in particular various diagonalizable and unipotent subgroups, provide important examples of dynamical systems on non-compact spaces, and many properties of orbits can be understood by studying the behavior of the associated discrete sets.

In this paper we describe an elementary, axiomatic approach to understanding quantitative results on excursions of orbits of diagonal and unipotent orbits away from compact sets on what we call \emph{Minkowski systems}, which abstract the common properties of the examples described above. This approach is inspired by that of~\cite{Weiss}.

Applying our results to particular cases, we  obtain quantitative results for rates of cusp excursions of geodesic and horocycle flows on the unit tangent bundle of hyperbolic surfaces $SL(2, \R)/\Gamma$. These yield non-trivial lower bounds on the deviation of ergodic averages for horocycle flows. When our results are applied to $X_n$, and various bundles over $X_n$,  we obtain information on both homogeneous and inhomogneous Diophantine approximation for systems of linear forms. We obtain `Diophantine'-type results for flows on $\q_g$ and various bundles over $\q_g$.

\subsection{Something old, something new, something borrowed.}\label{sec:borrowed} Like the old proverb about brides, mathematical papers tend to contain `something old, something new, something borrowed...' (the simile seems to break down at `something blue'). This paper grew out of the author attempting to push something old (\cite[Proposition 3.1]{AM1}) as far as possible, in particular to see if it could be used to obtain a logarithm law for horocycle flow on the moduli space of quadratic differentials. After many discussions with many people (see \S\ref{ack}), and a lot of simple linear algebra, the answer was that one could, and in fact one could obtain much more. In the course of this, it turns out that the author rediscovered certain simple ideas, particularly from the beautiful papers of Dani~\cite{Dani}, Kleinbock~\cite{Kleinbock:GAFA}, and Kleinbock-Margulis~\cite{KleinMarg}. Of course, the \emph{raison d'etre} for a paper is to contribute something new. In the author's view, the main novelty of this paper is contained in the following three ideas:

\subsubsection{Dani correspondence for horospherical flows} In retrospect, this paper can be seen as explaining a sort of Dani correspondence for horospherical flows on general moduli spaces, relating diophantine exponents and certain dynamical quantities. This is complementary to the above mentioned works ~\cite{Dani}, \cite{Kleinbock:GAFA}, and~\cite{KleinMarg} in which this correspondence was developed for diagonal flows. To keep our exposition as self-contained as possible, we have included proofs in the diagonal setting whenever possible.

\subsubsection{Logarithm Laws for Horocycle Flows on Moduli Space} In \S\ref{subsubsec:qd}, we obtain logarithm laws (Corollaries~\ref{cor:qd:mm} and ~\ref{cor:qd:mm:cyl}) for the horocycle flow on moduli spaces of quadratic differentials. This is a complementary result to Masur's~\cite{Masur:loglaw} logarithm law for the Teichm\"uller geodesic flow.

\subsubsection{Lower bounds for deviation of ergodic averages} In \S\ref{subsubsec:hyp}, we discuss how logarithm laws for horocycle flow on finite-volume non-compact hyperbolic surfaces can be used to obtain simple proofs of  non-trivial (but non-optimal) lower bounds for the deviation of ergodic averages for horocycle flows.

\subsection{Organization} This paper is structured as follows:  in \S\ref{sec:sl2}, we state our results for the motivating example of $X_2$, as a guide to our more general theorems, which we state in \S\ref{sec:results}. In \S\ref{sec:abs}, we formulate an abstract result (Theorem~\ref{theorem:abstract}) which lies at the heart of our proof, and use it to prove Theorem~\ref{theorem:axiomatic}. In \S\ref{sec:hyp}, we show how to associate discrete sets of vectors in $\R^2$ to points hyperbolic surfaces. In \S\ref{sec:diophexp}, we give a direct proof of the main results from \S\ref{sec:sl2}, and use them as a model to prove Theorem~\ref{theorem:abstract}.

\subsection{Acknowledgements}\label{ack} The ideas for this paper, which grew out of joint work with G.~A.~Margulis~\cite{AM1}, were initially developed on a visit in June 2009 to the Universities of Warwick, Bristol, and East Anglia, supported by the London Mathematical Society. We gratefully acknowledge those institutions for their hospitality and the LMS for making the trip possible.  Many thanks are also due to A.~Ghosh, J.~Marklof, and M.~Pollicott for arranging the LMS grant. Barak Weiss helped explain Lemma~\ref{lemma:john} and Proposition~\ref{prop:hyp:mink}. Thanks also to Y.~Cheung, G.~A.~Margulis, Y.~Minsky, and U.~Shapira for useful discussions. We would like to thank the organizers of the conference on Ergodic Geometry in Orsay in May 2011, where these results were presented, and the subsequent comments from participants greatly improved the paper. We would also like to profusely thank the anonymous referee of the initial version of this paper for drawing our attention to many previous results and also for improving the exposition significantly.

\section{Statement of Results}\label{sec:results} \noindent  In this section, we state our main result, Theorem~\ref{theorem:axiomatic}. We describe applications of this abstract viewpoint to systems of linear forms, quadratic differentials, and hyperbolic surfaces.

\subsection{Minkowski Systems}\label{sec:mahler:minkowski} A \emph{Minkowski} system is a non-compact topological space $X$ equipped with an action of $G = SL(d, \R)$ (with $d \geq 2$) action and a  $G-$equivariant assignment $x \mapsto \Lambda_x \subset \R^d\backslash \{0\}$, where $\Lambda_x$ is infinite and discrete, and moreover, that for all $x \in X$, $\La_x$ satisfies the \emph{Minkowski condition}: there is a$C = C_x$ so that for any $K \subset \R^d$ convex, centrally symmetric, $\mbox{vol}(K)>C$ $\exists \vv$ such that $$\vv \in \Lambda_x \cap K.$$

In many (but not all) of our examples, we also have that the assignment satisfies the following compactness condition, which we call the \emph{Mahler} condition: $A \subset X$ is precompact if and only if $\exists \epsilon_0>0$ such that for all $x \in A$, for all $\vv \in \Lambda_x$, $$\| \vv \| > \epsilon_0.$$ We call systems that satisfy both the Mahler and Minkowski conditions \emph{Mahler-Minkowski} systems. The Mahler condition is not crucial for our results, but allows us to interpret our results in terms of excursions away from compact sets.

\subsubsection{Cusp excursions}\label{subsubsec:cusp} We are interested in the excursions of orbits of various subgroups of $G$ to the subsets $X$ corresponding to the existence of short vectors in the associated discrete set. As mentioned above, in the setting when $X$ is a Mahler system, this can be interpreted in terms of excursions away from compact sets. Let $m, n \in  \N$, be so that $d = m+n$. Write $\R^d = \R^m \times \R^n$, and let $p_1$ and $p_2$ be the associated projections. We say a vector $\vv \in \R^d$ is \emph{vertical} (respectively \emph{horizontal}) if $p_1(\vv) = 0$ (respectively $p_2(\vv) = 0$).

\noindent Let $\| \cdot \|$ be the norm on $\R^d$ given by $\| \vv \| := \max (\|p_1(\vv)\|_2, \|p_2(\vv)\|_2)$ where $\|\cdot \|_2$ denotes the standard Euclidean norm. Define
\begin{equation}\label{eq:alpha1:def}\alpha_1(x) : = \sup_{\vv \in \Lambda_x} \frac{1}{\|\vv\|}\end{equation}

\noindent The Mahler condition can then be restated as saying that $A \subset X$ is precompact if and only if $\alpha_1|_A$ is bounded. The subgroups we consider are the one-parameter diagonal subgroup \begin{equation}\label{eq:gt:def}g_t : = \left(\begin{array}{cc}e^{nt/d} I_m & 0 \\0 & e^{-mt/d} I_n\end{array}\right).\end{equation}
\noindent and the associated horospherical subgroup
\begin{equation}\label{eq:h:def} H =\left \{h_B = \left(\begin{array}{cc}I_m & 0 \\B &  I_n\end{array}\right): B \in M_{n \times m}(\R)\right\}.\end{equation}
\noindent Note that $$g_t h_{B} g_{-t} = h_{e^{-t}B}.$$

\noindent We will measure our cusp excursions as follows: for $t, s>0$, let  \begin{equation}\label{eq:beta:def}\beta_s(x) = \sup_{ h \in \B_s} \log \alpha_1(h x),\end{equation} where $\B_s : = \{h_B: \|B\|_2 \le s\},$ and let
\begin{equation}\label{eq:gamma:def}\gamma_t(x) = \log \alpha_1 (g_t x)\end{equation}
Let \begin{equation}\label{eq:beta:lim:def} \beta(x) = \limsup_{s \rightarrow \infty} \frac{\beta_s(x)}{\log s}\end{equation} and
\begin{equation}\label{eq:gamma:lim:def} \gamma(x) = \limsup_{t \rightarrow \infty} \frac{\gamma_t(x)}{t}\end{equation}

\noindent Note that since all norms are equivalent up to a multiplicative constant, both $\gamma(x)$ and $\beta(x)$ are \emph{independent} of the norm chosen. We have specified the norm since it will be used in our proof. Also note that if $\vv \in \R^d$ is vertical, it is fixed by $h_B$ and contracted by $g_t$, and so if $\La_x$ has vertical vectors, $\gamma(x) = \frac{m}{d}$, and by the Mahler criterion, the $g_t$ orbit of $x$ (and $h x$ for all $h \in H$) is divergent (i.e., leaves every compact set). In many of our applications, the $H$-orbit of $x$ will be compact, yielding $\beta(x) = 0$. 

\subsubsection{Diophantine exponents}\label{subsubsec:dioph} We first define a notion of Diophantine exponents for arbitrary subsets of $\R^2$, and then use it to define a general notion. Let $\Omega \subset \R^2$ be a discrete subset of $\R^2$ without vertical vectors. We say that $\nu >0$ is an \emph{exponent} of $\Omega$ if there is a sequence of vectors $\left(\begin{array}{c}x_j \\y_j\end{array}\right) \in \Omega$ and a $C >0$ such that $|y_j| \rightarrow \infty$ and
\begin{equation}\label{nu2}|x_j| < C |y_j|^{-(\nu -1)}\end{equation}

\noindent That is, there is a sequence of vectors approximating the vertical direction at a rate specified by (\ref{nu2}). Let $\Exp(\Omega)$ denote the set of Diophantine exponents of $\Omega$. We define the \emph{Diophantine exponent} $\mu(\Omega) \in \R \cup \{+\infty\}$ by
\begin{equation}\label{mu2} \mu(\Omega) : = \sup \Exp(\Omega)\end{equation}

\noindent Suppose now that $\La_x$ does not have vertical vectors. Consider the map $\Pi_{m,n}: \R^d \rightarrow \R^2$ given by 
\begin{equation}\label{eq:Pi:def} \Pi_{m,n} (\vv) = (\|p_1(\vv)\|_2, \|p_2(\vv)\|_2)\end{equation}

\noindent Define the set $\Exp(x) = \Exp_{m, n} (x)$ of $(m,n)$-\emph{exponents} of $\Lambda_x$ as the set of exponents of $\Pi_{m,n}(\La_x)$. The set $\Exp(x)$ is the set of all $\nu >0$ so that there is a sequence of vectors $\{\vv_k\} \subset \Lambda_x$ and a $C>0$ satisfying 
$$\|p_1(\vv_k)\|_2 \le C\|p_2(\vv_k)\|_2^{-(\nu-1)}.$$ We will see in \S\ref{subsec:abstract:proof} that the Minkowski condition implies that $\frac{d}{m} \in \Exp(x)$ for all $x \in X$. Let $$\mu(x) = \sup \Exp(x).$$

\begin{Theorem}\label{theorem:axiomatic} For all $x \in X$ such that $\La_x$ does not have vertical vectors
\begin{equation}\label{eq:mm:beta:gamma} \beta(x) = \frac{n}{d} + \gamma(x) = 1- \frac{1}{\mu(x)}\end{equation}

\end{Theorem}

\medskip
\noindent In the remainder of this subsection, we record a few contextual remarks on Theorem~\ref{theorem:axiomatic}

\subsubsection{Vertical Vectors} There are two natural options for defining the exponent $\mu$ when $\La_x$ has vertical vectors. One could define $\mu(x) = \infty$, which would preserve the equality $ \frac{n}{d} + \gamma(x) = 1- \frac{1}{\mu(x)}$, or could define $\mu(x) =1$, which would preserve the equality $\beta(x) = 1- \frac{1}{\mu(x)}$. Due to this ambiguity, we avoid the setting of discrete sets with vertical vectors. 

\subsubsection{Related and Prior Results}This theorem, relating the rates of cusp excursions for diagonal and horospherical actions, is closely related to the main result in~\cite{AM2}, in which the author and G.~Margulis consider this relationship for actions on non-compact finite volume homogeneous spaces $G/\Gamma$. The result in \emph{loc. cit.} is more general in terms of the types of distance functions considered, and involves a careful study of reduction theory, and techniques from ergodic theory. 

As discussed in \S\ref{sec:borrowed}, the main novelty in our theorem is in the setting of the first equality, that is, for the horospherical action. The second equality, for the diagonal action, essentially follows from~\cite[Lemma 2.1]{Kleinbock:GAFA}, which is in turn based on earlier arguments in~\cite[\S8.5, \S9.2]{KleinMarg}. These are refined examples of the \emph{Dani correspondence}, introduced in Dani~\cite[\S2]{Dani}. Since the spirit of this paper is to show that the ideas involved are simple and linear-algebraic in nature, we have included complete proofs of both equalities.

\subsubsection{Strategy of Proof}The proof of Theorem~\ref{theorem:axiomatic} is axiomatic and elementary, relying only on linear algebra. It applies, as discussed above and seen below, to a variety of natural geometric contexts, and provides results for \emph{every} orbit, in contrast to ergodic theoretic methods. Many of the ideas come from the setting where $\La \subset \R^2$, and for the sake of exposition, we devote \S\ref{sec:sl2} to it, in particular to the specific setting of lattices. The main idea of the proof is really that of~\cite[Proposition 3.1]{AM1}, which really boils down to the following simple observation: \textit{the vector $$\left(\begin{array}{cc}1 & 0 \\-s & 1\end{array}\right)\left(\begin{array}{c}x \\y\end{array}\right) = \left(\begin{array}{c}x \\y-sx\end{array}\right)$$ is shortest at time $s = \frac{y}{x}$, and has length $x$ at that time. Thus, if $x$ is very small relative to $y$ (that is, the original vector is close to the vertical), at time $s$, the vector will be very short relative to the time $s$. }

\subsection{Markoff and Minksowski constants}\label{sec:markoff:minkowski} We obtain finer information if we consider simply $\alpha_1$ instead of $\log \alpha_1$. For $s >0, x \in X$, let $$\sigma_s(x) : = \sup_{b \in \B_s} \alpha_1(h_b \Lambda_x) = e^{\beta_s(x)} ,$$ and $$\sigma(x) : = \limsup_{s \rightarrow \infty} \frac{\sigma_s(x)}{s^{\frac{n}{d}}}.$$ 
\noindent Note that $\sigma(x)$ does depend on our choice of norm. Given $\Omega \in \R^2$ without vertical vectors, we define the $(m,n)$-\emph{Markoff constant} $\tilde{c} (\Omega)$ as the infimum of the set of $\tilde{c} >0$ so that there exist a sequence $\{(x_k, y_k)^T\} \subset \Omega$ with $$|x_k| \le \tilde{c} |y_k|^{-\frac{n}{m}}.$$

\noindent Define the $(m,n)$-\emph{Markoff constant} $\tilde{c} (x) = \tilde{c}(\Pi_{m,n}(\La_x))$. This is the infimum of the set of $\tilde{c} >0$ so that there exist a sequence $\{\vv_k\} \subset \Lambda_x$ with $$\|p_1(\vv_k)\| \le \tilde{c} \|p_2(\vv_k)\|^{-\frac{n}{m}}.$$

\noindent Denote by $c(x)$ the minimal constant in the Minkowski property for $\La_x$ (i.e.,  the supremum of the volume of a convex centrally symmetric set that does not intersect $\La_x$). We call $c(x)$ the \emph{Minkowski constant} of $x$. The Markoff constant is well-defined (and finite) because of the following relation, which we will prove in \S\ref{subsec:abstract:markoff:proof}:
$$c(x) \geq \tilde{c}(x)^m a_m a_n,$$ where $a_j$ denotes the volume of the unit (Euclidean norm) ball in $\R^j$.

\begin{Prop}\label{prop:markoff:mink} For all $x \in X$ such that $\La_x$ contains no vertical vectors, $$\sigma(x) = \tilde{c}(x)^{-\frac{m}{d}}.$$ When $\tilde{c}(x) = 0$ this is taken to mean $\sigma(x) = \infty$.
\end{Prop}

\subsection{Systems of Linear Forms}\label{subsubsec:lattices} The motivating example of a Mahler-Minkowski system is the space of lattices $SL(d, \R)/SL(d, \Z)$ endowed with its natural $SL(d, \R)$ action. Here, the assignment is simply $g SL(d, \Z) \mapsto g\Z^d_{prim}$, where $\Z^d_{prim}$ denotes the set of non-zero primitive integer vectors. The classical Mahler compactness criterion and Minkowski convex body theorems yield our conditions, and thus Theorem~\ref{theorem:axiomatic} applies here, giving a relation between cusp excursions and Diophantine exponents for lattices. Note that for each decomposition $d = m+n$ we obtain results relating the associated diagonal and horospherical flows to a different diophantine exponent. 

As an application of our results, let $A \in M_{m \times n}(\R)$ be an $m \times n$ matrix, and define the \emph{Diophantine} exponent $\mu(A)$ by the supremum of the set of  $\nu >0$ so that there are infinitely many $\vv \in \Z^d$ such that $$ \|Ap_2(\vv) - p_1(\vv)\|_2 \le \|p_2(\vv)\|_2^{-(\nu-1)}$$

\noindent A natural interpretation of $A$ is as a system of $m$ linear forms in $n$ variables. The exponent $\mu(A)$ measures the degree of closeness of  approximate integer solutions $\vv \in \Z^d$ to the system $A p_1(\vv)= p_2(\vv)$. Let $$x_{A} : = \left(\begin{array}{cc}I_m & -A \\0 &  I_n\end{array}\right) SL(d, \Z).$$ The associated lattice $\La_{x_A}$ has vertical vectors if and only if the equation $A p_2(\vv) = p_1(\vv)$ has a non-zero solution $\vv \in \Z^d$. The following theorem, relating the behavior of cusp excursions of the orbit of $x_A$ to $\mu(A)$, is a direct corollary of Theorem~\ref{theorem:axiomatic}:

\begin{Theorem}\label{theorem:dioph:matrix} For any $A \in M_{m \times n}(\R)$, so that $A p_2(\vv) \neq p_1(\vv)$ for all $\vv \in \Z^d\minuszero$,
$$\beta(x_A) = \frac{n}{d} + \gamma(x_A) = 1 - \frac{1}{\mu(A)}.$$

\end{Theorem}

\medskip 
\noindent The easy half of the Borel-Cantelli lemma shows that $\mu(A) le \frac{m}{d}$ for Lebesgue almost every matrix $A$, and Dirichlet's theorem shows $\mu(A) \geq \frac{m}{d}$ for all irrational $A$. Thus, $\mu(A) = \frac{m}{d}$ almost everywhere, and as a corollary, we obtain that $$\gamma(x_A) = 0$$ (which als and $$\beta(x_A) = \frac{n}{d}$$ for almost every $A$. This observation, and the second equality, were both originally made in~\cite{KleinMarg}.
\subsubsection{Inhomogeneous Systems}\label{subsec:dioph} We can also consider the space of \emph{affine lattices} $(SL(d, \R) \semidirect \R^d )/( SL(d, \Z) \semidirect \Z^d)$, and more generally the fiber bundles $(SL(d, \R) \semidirect (\R^d)^k)/( SL(d, \Z) \semidirect (\Z^d)^k)$ of tori with $k$ marked points. Here, the discrete set associated to $x = (g; \vv_1, \vv_2, \ldots, \vv_k) (SL(d, \Z) \semidirect (\Z^d)^k)$ is given by $\Lambda_x = \bigcup_{i=1}^k (g\Z^d + \vv_i)$ (i.e., the union of $k$ affine lattices). As above, the classical Minkowski theorems guarantee that Theorem~\ref{theorem:axiomatic} applies.

The case $k=1$ has a nice application to inhomogeneous linear forms. As above, let $A$ be an $m \times n$ matrix, and let $\vv_0 \in [0, 1)^d = \R^d /\Z^d$. We would like to find approximate integer solutions $\vv \in \Z^d$ to the system of equations  
$$Ap_2(\vv) = p_1(\vv + \vv_0).$$

\noindent We assume that there are no \emph{exact} integer solutions to the system, and define the \emph{Diophantine} exponent $\mu(A, \vv_0)$ by the supremum of the set of  $\nu >0$ so that there are infinitely many $\vv \in \Z^d$ such that $$ \|Ap_2(\vv) - p_1(\vv+\vv_0)\|_2 \le \|p_2(\vv)\|_2^{-(\nu-1)}$$

\noindent Let $$x_{A, \vv_0} =  \left(\left(\begin{array}{cc}I_m & -A \\0 &  I_n\end{array}\right) ; \vv_0\right)SL(d, \Z) \semidirect \Z^d$$

\begin{Theorem}\label{theorem:dioph:form} For any $A \in M_{m \times n}(\R)$, $\vv_0 \in (0, 1)^d$ so that $Ap_2(\vv) \neq p_1(\vv+\vv_0)$ for all $\vv \in \Z^d$, 
$$\beta(x_A, \vv_0) = \frac{n}{d} + \gamma(x_{A, \vv_0}) = 1 - \frac{1}{\mu(A, \vv_0)}.$$
\end{Theorem}

\subsection{Quadratic Differentials}\label{subsubsec:qd}

An important novelty of our approach is that it can be applied directly to study cusp excursions on the moduli space of quadratic differentials. Here, the acting group is $SL(2, \R)$, and our subgroups of interest are
$$g_t : = \left(\begin{array}{cc}e^{t/2} & 0 \\0 & e^{-t/2} \end{array}\right),$$
\noindent and 
$$H =\left \{h_s = \left(\begin{array}{cc}1 & 0 \\s &  1\end{array}\right): s \in \R\right\}.$$

We recall briefly some background on quadratic differentials. Let $M$ be a Riemann surface. A (holomorphic) quadratic differential $q$ on $M$ is a tensor of the form (in local coordinates) $f(z)dz^2$, where $f$ is holomorphic. A quadratic differential determines a singular flat metric on $M$ with singularities at the zeros of $q$. A \emph{saddle connection} is a geodesic segment connecting two singularities with no singularities in its interior. The \emph{holonomy vector} $\vv_{\gamma} \in \C$ of a saddle connection $\gamma$ is given by integrating a (local) square root of the form $q$ along the saddle connection. This is well defined up to a choice of sign. Given a quadratic differential $q$, define the associate set of holonomy vectors:
$$\La_q : = \{\vv_{\gamma}: \gamma \mbox{ a saddle connection on } q\}$$

\noindent We view $\La_q$ as a subset of $\R^2$. We include both choices of holonomy vectors in $\La_q$. The set $\La_q$ is discrete (see, e.g., ~\cite[Proposition 3.1]{Vorobets96}). 

Fix $g \geq 2$. Let $Q_g$ denote the moduli space of \emph{unit area} genus $g$ quadratic differentials, that is, the space of pairs $(M, q)$ where $M$ is a compact genus $g$ Riemann surface and $q$ a holomorphic quadratic differential on $M$. Two pairs $(M_1, q_1)$ and $(M_2, q_2)$ are equivalent if there is a biholomorphism $f: M_1 \rightarrow M_2$ so that $f_* q_1 = q_2$.  We will refer to points in $Q_g$ as $q$, with the Riemann surface $M$ implicit.

The sum of the orders of the zeros of a quadratic differential $q \in Q_g$ is $4g-4$, and the space $Q_g$ can be stratified by integer partitions of $4g-4$. The stratum $\hh$ associated to a partition $(\alpha_1, \ldots \alpha_k)$ consists of the quadratic differentials with $k$ zeros of orders $\alpha_1, \ldots, \alpha_k$ respectively. Strata are not always connected, but Kontsevich-Zorich~\cite{KZ} and Lanneau~\cite{Lanneau}  have classified the connected components.  Most strata are connected, and there are never more than three connected components. 

There is a natural $SL(2, \R)$ action on $Q_g$ which preserves this stratification: a quadratic differential $q$ determines (and is determined by) an atlas of charts on the surface away from the singular points to $\C$ whose transition maps are of the form $z \mapsto \pm z + c$. Identifying $\C$ with $\R^2$, we have an $SL(2, \R)$ action given by (linear) postcomposition with charts. The assignment $q \mapsto \La_q$ gives an $SL(2, \R)$ equivariant assignment of a discrete set.

\begin{Theorem}\label{theorem:qd:mm}Let $X$ be a connected component of a stratum.  The assignment $q \mapsto \La_q$ gives $X$ the structure of a $SL(2, \R)$-Mahler-Minkowski system. \end{Theorem}
\begin{proof} The Mahler criterion follows from, e.g.,~\cite[Proposition 1, \S3]{KMS}, and the Minkowski condition from~\cite[Theorem 1]{HS}. 
\end{proof}
\medskip
\noindent If $\La_q$ has a vertical vector, then the Mahler criterion shows that the orbit under $g_t$ is divergent. Thus, we assume $\La_q$ does \emph{not} have vertical vectors. Using $\La_q$, we define the notion of the (saddle connection) \emph{Diophantine exponent} $\mu(q)$ of a quadratic differential $q$ as the supremum of the set of $\nu >0$ so that there exist an sequence of saddle connections $\gamma_k$ on $q$ such that $\vv_k = \vv_{\gamma_k} = (x_k , y_k)^T$ satisfy 
$$|x_k| \le |y_k|^{-(\nu-1)}$$
Theorem~\ref{theorem:qd:mm} allows us to apply Theorem~\ref{theorem:axiomatic} and Proposition~\ref{prop:markoff:mink} to the setting of quadratic differentials. To distinguish from the space of lattices, for $q \in Q_g$, we write  $$l(q) = \sup_{\vv \in \La_q }\frac{1}{\|\vv\|}$$ Let $$\gamma(q) = \limsup_{t \rightarrow \infty} \frac{ \log (l(g_t q)} {t} , $$ $$\beta(q) = \limsup_{|s| \rightarrow \infty} \frac{\log(l(h_s q))}{\log |s|}, $$ $$\sigma(q) =  \limsup_{|s| \rightarrow \infty} \frac{l(h_s q)}{|s|^{\frac{1}{2}}}.$$As above, let $c(q) = c(\La_q)$ denote the Minkowski constant of $q$.

\begin{Cor}\label{cor:qd:mm} Let $q \in Q_g$, and suppose $\La_q$ does not have vertical vectors. Then
$$\beta(q) = \frac{1}{2} + \gamma(q) = 1- \frac{1}{\mu(q)}$$ and
$$\sigma(q) = c(q)^{-\frac{1}{2}}.$$
\end{Cor}

\begin{remark*} Masur's logarithm law~\cite{Masur:loglaw} implies that for almost every $q$, $\gamma(q)=0$, yielding that $\mu(q) = 2$ and $\beta(q) =0$ almost everywhere in $Q_g$. \end{remark*}

\subsubsection{Periodic cylinders}\label{subsec:cylinder} We can also associate the set of \emph{periodic cylinder} holonomy vectors to a quadratic differential $q \in Q_g$. This is the set $\La^{cyl}_q$ of holonomy vectors of periodic geodesics which do \emph{not} intersect a singular point. In this case, we need to consider the total space $Q_g$ as opposed to a stratum $X$. We have the following:

\begin{Theorem}\label{theorem:qd:cyl} The assignment $q \mapsto \La^{cyl}_q$ gives $Q_g$ the structure of a $SL(2, \R)$-Mahler-Minkowski system. \end{Theorem}

\medskip

\noindent Defining the \emph{cylinder exponent} $\mu_{cyl}(q)$, and the functions $l_{cyl}$, $\beta_{cyl}$ and $\gamma_{cyl}$ as above, we obtain the natural analogue of Corollary~\ref{cor:qd:mm} for periodic cylinder approximation. 

To prove Theorem~\ref{theorem:qd:cyl}, we first note the Mahler property is simply a restatement of the Mumford compactness criterion, so we prove only the Minkowski property. For this proof, we follow the outline of the proof of~\cite[Theorem 1]{HS}. We first state an abstract lemma, whose proof is essentially contained in~\cite[\S3.2]{HS} (which we will use again in \S\ref{sec:hyp:mink}):

\begin{lemma}\label{lemma:john} $\Omega \subset \R^2\minuszero$ satisfies the Minkowski property if and only if there is a $R_0 >0$ so that for all $g \in SL(2, \R)$, $$g\Omega \cap B(0, R_0) \neq \emptyset.$$
\end{lemma}

\begin{proof} First suppose $\Omega$ satisfies the Minkowski property. Suppose there is a sequence $R_n \rightarrow \infty$ and $g_n \in SL(2, \R)$ so that $g_n \Omega \cap B(0, R_n) = \emptyset$. Then $\Omega \cap g_n^{-1}B(0, R_n) = \emptyset$, which is a contradiction.

For the converse, let $c = 2\pi R_0^2$. Let $M$ be convex, centrally symmetric, and of volume at least $c$.  The ellipse $E$ (centered at $0$) of maximal area contained in $M$ has at least half the area of $M$, (this is a theorem of Fritz John (see, e.g. K.~Ball's survey~\cite{Ball}). There is an element $g \in SL(2, \R)$ and a $R>0$ so that $g E = B(0, R)$. Since the area of $E$ is at least $\pi R_0^2$, we have $R \geq R_0$, and so $g E \cap g \Omega \neq \emptyset$, and so $E \cap \Omega \neq \emptyset$.
\end{proof}

\medskip

\noindent Theorem~\ref{theorem:qd:cyl} now follows from a result of Masur~\cite{Masur}, who showed that there is a constant $C_g$ so that for all $q \in Q_g$, there is a periodic cylinder of length at most $C_g$ on $q$. We also note that the best known bounds for $C_g$ are due to Vorobets~\cite{Vorobets2}, which will give us the best known bounds on the Minkowski contsant for $\La^{cyl}_q$.

Finally, defining $\mu_{cyl}(q), \beta_{cyl}(q), \gamma_{cyl}(q), c_{cyl}(q)$, and $\sigma_{cyl}(q)$ as the associated quantities for $\La_{cyl}(q)$, we have 

\begin{Cor}\label{cor:qd:mm:cyl} Let $q \in Q_g$, and suppose $\La^{cyl}_q$ does not have vertical vectors. Then
$$\beta_{cyl}(q) = \frac{1}{2} + \gamma_{cyl}(q) = 1- \frac{1}{\mu_{cyl}(q)}$$ and
$$\sigma_{cyl}(q) = c(q)_{cyl}^{-\frac{1}{2}}.$$
\end{Cor}

\begin{remark*} As above, Masur's logarithm law~\cite{Masur:loglaw} implies that for almost every $q$, $\gamma_{cyl}(q)=0$, yielding that $\mu_{cyl}q) = 2$ and $\beta_{cyl}(q) =0$ almost everywhere in $Q_g$. \end{remark*}

\subsection{Hyperbolic Surfaces}\label{subsubsec:hyp} Let $\Gamma \subset SL(2, \R)$ be a \emph{non-uniform lattice}, that is, let $\Gamma$ be a finite-volume, non-compact discrete subgroup of $SL(2, \R)$. $X = SL(2, \R)/\Gamma$ is the unit-tangent bundle of the finite-volume non-compact hyperbolic orbifold $\h^2/\Gamma$. Let $g_t$ and $h_s$ be as in \S\ref{subsubsec:qd}. Let $d(\cdot, \cdot)$ denote the hyperbolic metric on $\h^2/\Gamma$ (we normalize so that the curvature on $\h^2$ is $-1$). By abuse of notation, given $x, y \in X$, we will write $d(x, y)$ for the distance between their projections to $\h^2/\Gamma$. Given $x \in X$, we will write $B(x, R)$ for the collection of $y \in X$ so that $d(x, y) < R$. Given $x = g\Gamma$, we have that $Hx$ is closed in $X$ if and only if $\{g_t x\}_{t \geq 0}$ is divergent in $x$. We call such $x$ \emph{cuspidal}. An application of our main result is as follows:

\begin{Theorem}\label{theorem:hyp:excursions} Fix $x_0 \in X$, and suppose $x$ is not cuspidal. Then 
$$\limsup_{|s| \rightarrow \infty} \frac{d(h_s x, x_0)}{\log |s|} = 1 + \limsup_{t \rightarrow \infty} \frac{d(g_t x, x_0)}{t} $$
\end{Theorem}

\begin{remark*} A particular strength of this theorem is that it holds even for $x$ taken from the set of measure $0$ where $\gamma(x) = \limsup_{t \rightarrow \infty} \frac{d(g_t x, x_0)}{t} >0$. The fact that $\beta(x) = \limsup_{|s| \rightarrow \infty} \frac{d(h_s x, x_0)}{\log |s|}= 1$ almost everywhere has recently been generalized to the setting of quotients of products of $SL(2,\R)$ and $SL(2, \C)$ by Kelmer-Mohammadi~\cite{KelMoh}.\end{remark*}

\medskip\noindent In fact, we will be able to prove results about excursions to each individual cusp of $X = SL(2, \R)/\Gamma$ by studying different families of discrete sets. Theorem~\ref{theorem:hyp:excursions} will be a corollary of these results obtained by taking the union of the discrete sets associated to each cusp. Our discrete sets will be associated to discrete orbits of the linear $\Gamma$ action on $\R^2\minuszero$, which we will view as the space of horocycles on $G = SL(2, \R)$. Such discrete orbits exist for a lattice if and only if it is nonuniform (cf.~\cite{Ledrappier}). In \S\ref{sec:hyp}, we will carefully define our assignment and prove that it satisfies a \emph{quantitative} Mahler condition, in which the length of short vectors will be precisely related to hyperbolic distance. We do not have a proof of the Minkowski condition in this setting, but we will not require it, as other geometric considerations will allow us to circumvent it.  

\subsubsection{Deviation of ergodic averages}\label{subsec:dev} Theorem~\ref{theorem:hyp:excursions} yields a non-trivial lower bound on the deviation of ergodic averages for the horocycle flow on $X = SL(2, \R)/\Gamma$. While it does not match the results obtained by Strombergsson~\cite{Strombergsson}, it provides non-trivial information using elementary techniques, in particular, with no reference to the eigenvalues of the hyperbolic Laplacian on $X$. Let $\mu$ denote the probability measure on $X$ given by (normalized) Haar measure on $SL(2, \R)$.

\begin{Theorem}\label{theorem:hyp:deviations} Given $x_0, x \in X$ so that $x$ is not cuspidal, let $$\gamma = \gamma(x) : =  \limsup_{t \rightarrow \infty} \frac{d(g_t x, x_0)}{t}.$$ Suppose $\gamma >0$. Then for all $S>0, \epsilon \in (0, \gamma), \delta \in (0, 1)$, there is a $s_0 > S$ and $\kappa >0$ so that, writing $C = B(x_0, \kappa \log s_0)$, 
\begin{equation}\label{eq:hyp:deviations}
\left| \int_{-s_0}^{s_0} \chi_C( h_s x) ds  - 2s_0 \mu (C)\right| \geq s_0^{\gamma - \epsilon} (1- \delta)
\end{equation}
\end{Theorem} 

\medskip

\noindent We will prove this theorem in \S\ref{sec:devproof}. A main weakness of the result is that it does not apply to a single, fixed compact set but rather a sequence of growing targets. However, the shape of these sets is specified, which makes it a non-trivial result. The results of~\cite{Strombergsson} are considerably more delicate and sophisticated, but as mentioned above, require much more detailed analysis.

\section{The Motivating Example}\label{sec:sl2}

\noindent To illustrate our approach, we consider our base example $$X_2 = SL(2, \R)/SL(2, \Z).$$ 

\subsection{Interpretations of $X_2$}

$X_2$ can be identified the space of unimodular lattices in $\R^2$ via
$$gSL(2, \Z) \leftrightarrow g\Z^2$$
\noindent Further identifying 
$$g\Z^2 \leftrightarrow \R^2/g\Z^2$$
\noindent we can view $X_2$ as the space of unit-area flat tori with a choice of direction (the vertical in $\R^2$). Finally, identifying
$$\R^2/g\Z^2 \leftrightarrow (\C/g\Z^2, (dz)^2),$$ 

\noindent we identify $X_2$ with the space of unit-area holomorphic quadratic differentials on compact genus $1$ Riemann surfaces.

\subsection{Discrete sets}\label{sec:ident}

As above, let $\Z^2_{prim}$ denote the set of non-zero primitive vectors in $\Z^2$.
\begin{equation}\label{assign2} x = g SL(2, \Z) \leftrightarrow \Lambda_x = g \Z^2_{prim}\end{equation}

\noindent This assignment is $SL(2, \R)$-equivariant, and assigns to each coset the primitive vectors in the corresponding unimodular lattice, or equivalently the set of holonomy vectors of the (square root of the) differential integrated along simple closed curves. We recall the classical Mahler compactness criterion and Minkowski convex body theorem. 

\begin{Prop}[Mahler's compactness criterion]\label{prop:mahler:2}

$A \subset X_2$ is pre-compact if and only if there exists an $\epsilon_0 >0$ such that for all $x \in A$,  for all $v \in \Lambda_x$, $\|v\| \geq \epsilon_0$. Here, $\| \cdot \|$ is any norm on $\R^2$

\end{Prop}

\medskip\noindent Defining $\alpha_1: X_2 \rightarrow \R^+$ by $\alpha_1(x) : = \sup_{\vv \in \La_x} \frac{1}{\| \vv\|}$, Proposition~\ref{prop:mahler:2} says that $\alpha_1$ is unbounded off of compact sets (notice this is independent of the choice of norm used to define $\alpha_1$).

\noindent 
\begin{Prop}\label{prop:minkowski:2} Let $\La \subset \R^2$ be a unimodular lattice, and let $K \subset \R^2$ be a convex, centrally symmetric set of volume at least $4$. Then there is a (non-zero) vector
$$\mathbf{v} \in \La_{prim} \cap K.$$

\end{Prop}

\subsection{Dynamics and Diophantine approximation}\label{sec:dyndioph}

\subsubsection{Diophantine exponents in $\R^2$}\label{sec:diophexp:lattices}

\noindent Let $\La$ be a unimodular lattice. It is a standard exercise to use Minkowski's Theorem (Proposition~\ref{prop:minkowski:2}) to prove the following

\begin{lemma}\label{lemma:exp:2} Let $\La \subset \R^2$ be a unimodular lattice with no vertical vectors. Then 
\begin{equation}\label{eq:exp:2} \mu(\La) \geq 2\end{equation}

\end{lemma}

%\medskip
%
%\begin{proof} Let
% $$K_j : = \left\{\left(\begin{array}{c}x \\y\end{array}\right) \in \R^2: |x| \le \frac{1}{j}, |y| \le j\right\}.$$
%
%\noindent $K_j$ is convex, centrally symmetric, and has volume $4$, so by Proposition~\ref{prop:minkowski:2}, we can find 
%$$\left(\begin{array}{c}x_j \\y_j\end{array}\right) \in \La \cap K_j.$$
%
%\noindent By the definition of $K_j$, 
%$$|x_j| \le \frac{1}{j} \le \frac{1}{|y_j|} = |y_j|^{-(2-1)}.$$
%
%\noindent Since $\La$ is discrete, and $|x_j| \rightarrow 0$, $|y_j|$ must be bounded below, so $\frac{|y_j|}{|x_j|} \rightarrow \infty$. So we have $2 \in \Exp(\La)$, as desired.\end{proof}

Lemma~\ref{lemma:exp:2} allows us to define the \emph{Markoff constant} $\tilde{c}(\La)$ of a lattice $\La$ as the infimum of all $C$ so that there exist a sequence of vectors $\left(\begin{array}{c}x_j \\y_j\end{array}\right) \in \La$ satisfying (\ref{nu2}).

\subsubsection{Geodesic and horocycle flows}\label{sec:geod:horo}

There is a relation between the cusp excursions of the orbit of the point $x = gSL(2, \Z)$ under the geodesic and horocycle flows on $X_2$ and the diophantine exponent $\mu(x) = \mu(\La_x)$.

\begin{equation}\label{eq:geod}
g_t : = \left(\begin{array}{cc}e^{t/2} & 0 \\0 & e^{-t/2} \end{array}\right)
\end{equation}

\begin{equation}\label{eq:horo}
h_s = \left(\begin{array}{cc}1 & 0 \\s &  1\end{array}\right)
\end{equation}

\noindent The actions of these one-parameter subgroups of $SL(2,\R)$ give the geodesic and horocycle flows on $X_2$ respectively. In~\cite{AM1}, we studied the relationship between $\mu(g\Z^2)$ and the behavior of the orbits $\{g_t x\}$ and $\{h_s x\}$ for a particular collection of lattices $x$. The following result is a straightforward generalization of Proposition 3.1 from~\cite{AM1}.

\begin{Prop}\label{prop:dioph:2} Let $x \in X_2$ be such that $\La_x$ does not have vertical vectors. Then

\begin{equation}\label{eq:horo:lim}
\limsup_{|s| \rightarrow \infty} \frac{\log(\alpha_1(h_sx))}{\log |s|} = 1 - \frac{1}{\mu(x)} 
\end{equation}

\begin{equation}\label{eq:geod:lim}
\limsup_{t \rightarrow \infty} \frac{\log(\alpha_1(g_tx))}{t} = \frac{1}{2} - \frac{1}{\mu(x)} 
\end{equation}

\end{Prop}

\medskip

\begin{remark*} This is a special case of our main result Theorem~\ref{theorem:axiomatic}. It is a simple generalization of \cite[Proposition 3.1]{AM1}, which considered (with the roles of upper and triangular matrices reversed) the special case where $$g  = \left(\begin{array}{cc}1 & \alpha \\0 &  1\end{array}\right),$$ with $\alpha \notin \Q$. Then, $\mu(g\Z^2)$ coincides with the classical notion of Diophantine exponent of $\alpha$, and the Markoff constant $\tilde{c}(g\Z^2)$ coincides with the classical Markoff constant of $\alpha$.

\end{remark*}

\subsubsection{Markoff constants}\label{sec:mark:mink} We can also detect the Markoff constant using the horocycle flow. Fix $\|\cdot\|$ to be supremum norm on $\R^2$, and define $\alpha_1$ on $X_2$ using this norm. We have

\begin{Prop}\label{prop:mark:2} Suppose $x \in X_2$ is such that $\La_x$ does not contain vertical vectors. Then
\begin{equation}\label{eq:mark:2}
\limsup_{|s| \rightarrow \infty} \frac{\alpha_1(h_sx)}{|s|^{\frac{1}{2}}} = \tilde{c}(x)^{-\frac{1}{2}}
\end{equation}

\end{Prop}
\medskip
\noindent We will obtain this result as a corollary of our general result Proposition~\ref{prop:markoff:mink}. In \S\ref{sec:proof:prop:mark:2}, we will formulate \emph{one-sided} (i.e., $s \rightarrow \pm \infty$) versions of (\ref{eq:horo:lim}) and (\ref{eq:mark:2}).

\section{General Discrete Sets}\label{sec:abs}
In this section we state our main technical result Theorem~\ref{theorem:abstract} and use it to prove Theorem~\ref{theorem:axiomatic} in \S\ref{subsec:abstract:proof}. In \S\ref{sec:abstract:markoff} we prove Proposition~\ref{prop:markoff:mink} using a technical lemma Prop~\ref{prop:abstract:markoff}.

\subsection{Notation}\label{subsec:notation} We recall notation: let $m, n \in  \N$, and let $d = m+n$, so $\R^d = \R^m \times \R^n$. Recall that  $p_1: \R^d \rightarrow \R^m$ and $p_2: \R^d \rightarrow \R^n$ are the associated projections, and that we say that vectors in the kernel of $p_2$ are horizontal and those in the kernel of $p_1$ are vertical. We let $\| \cdot \| = \| \cdot \|_{m,n}$ be the norm on $\R^d$ given by $\| \vv \| := \max (\|p_1(\vv)\|_2, \|p_2(\vv)\|_2)$ where $\|\cdot \|_2$ denotes the standard Euclidean norm. Let $\Lambda \subset \R^d \minuszero = \R^m \times \R^n$ be discrete and without vertical vectors.  We define
$$\alpha_1(\Lambda) : = \sup_{\vv \in \Lambda} \frac{1}{\|\vv\|}$$
\noindent We define the set $\Exp(\Lambda) = \Exp_{m, n} (\Lambda)$ of $(m,n)$-\emph{exponents} of $\Lambda$ as the set of exponents of $\Pi_{m,n}(\Lambda)$, where $\Pi_{m,n}: \R^d \rightarrow \R^2$ is given by $\Pi_{m,n}(\vv) = (p_1(\vv), p_2(\vv))$. We will assume that this set is non-empty. Let $$\mu(\Lambda) = \mu_{m,n}(\Lambda) = \sup \Exp(\La).$$
\noindent Let $g_t$ and $h_B$ be as in (\ref{eq:gt:def}) and (\ref{eq:h:def}) respectively, and for $t, s>0$, let  $$\beta_s(\La) = \sup_{ h \in \B_s} \log \alpha_1(h \Lambda),$$ where $\B_s : = \{h_B: \|B\|_2 \le s\},$ and let
$\gamma_t(\La) = \log \alpha_1 (g_t \La).$
Let \begin{eqnarray}\label{eq:beta:la:def}\beta(\La) &=& \limsup_{s \rightarrow \infty} \frac{\beta_s(\La)}{\log s}  \\ \nonumber \gamma(\La) &=& \limsup_{t \rightarrow \infty} \frac{\gamma_t(\La)}{t}\end{eqnarray}

\begin{Theorem}\label{theorem:abstract} Fix notation as above. Let $\La \subset \R^d\minuszero$ be discrete and without vertical vectors, and $\beta = \beta(\La)$, $\gamma = \gamma(\La)$.

\begin{enumerate}

\item For all $\nu \in \Exp(\La)$, $\beta \geq 1-\frac{1}{\nu}$.

\item Suppose $\beta>0$. Then  $\left(1, \frac{1}{1-\beta}\right) \subset \Exp (\La)$.

\item For all $\nu \in \Exp(\La)$, $\gamma \geq \frac{m}{d}-\frac{1}{\nu}$

\item Suppose $\gamma \geq 0$. Then $\left(0, \frac{1}{\frac{m}{d}-\gamma}\right) \subset \Exp (\La)$.

\end{enumerate}

\end{Theorem}

\subsection{Proof of Theorem~\ref{theorem:axiomatic}}\label{subsec:abstract:proof} We show how Theorem~\ref{theorem:axiomatic} follows from Theorem~\ref{theorem:abstract}. We require the following analogue of Lemma~\ref{lemma:exp:2}:

\begin{lemma}\label{lemma:exp:d} Let $\La \subset \R^d\minuszero$ be discrete and satisfy the Minkowski condition. That is, suppose there is a $c >0$ so that for all convex, centrally symmetric sets $K \subset \R^d$ with volume at least $c$, $K \cap \La$ is non-empty. Then $$\mu_{m,n}(\La) \geq \frac{d}{m}$$\end{lemma}

\begin{proof} We will show $\frac{d}{m} \in \Exp(\La) = \Exp(\Pi_{m,n}(\La))$. Let $a_{k}$ denote the volume of the standard (Euclidean) unit ball in $\R^k$. Let $d$ be such that $d^m > \frac{c}{a_m a_n}$. For $j \in \N$, let
 $$K_j: = \{\vv \in \R^d: \|p_1(\vv)\| \le d/j, \|p_2(\vv)\| \le j^{\frac{m}{n}}\}.$$

\noindent $K_j$ is convex, centrally symmetric, and has volume $d^m a_m a_n > c$, so we can find
$$\vv_j \in \La \cap K_n.$$

\noindent By the definition of $K_j$, 
$$\|p_1(\vv_j)\| \le d/j \le d \|p_2(\vv_j)\|^{-\frac{n}{m}} =d \|p_2(\vv_j)\|^{-(\frac{d}{m}-1)}.$$

\noindent Since $\La$ is discrete, and $\|p_1(\vv_j)\| \rightarrow 0$, $\|p_2(\vv_j)\|$ must be bounded below, so $\frac{\|p_2(\vv_j)\|}{\|p_1(\vv_j)\|} \rightarrow \infty$. So we have $\frac{d}{m} \in \Exp(\La)$, as desired.\end{proof}
\medskip
\noindent Now let $X$ be an $SL(d, \R)$-Mahler-Minkowski system. Given $x \in X$ so that $\La_x$ does not have vertical vectors, the Minkowski condition and Lemma~\ref{lemma:exp:d} shows that $\frac{d}{m} \in \Exp(\La_x)$. Part (1) of Theorem~\ref{theorem:abstract} then shows that $\beta(x) \geq 1- \frac{1}{\mu(x)} > 0$ (since $\mu(x) \geq \frac{d}{m} >1$), and combining this with part (2), we have $$\beta(x) = 1- \frac{1}{\mu(x)}.$$ Similarly, combining part (3) and part (4) yields $$\gamma(x) = \frac{m}{d} - \frac{1}{\mu(x)},$$ completing the proof of Theorem~\ref{theorem:axiomatic}.\qed\medskip

\subsection{Abstract Markoff constants}\label{sec:abstract:markoff} With notation as in \S\ref{subsec:notation}, define $$\sigma_s(\La) : = \sup_{b \in \B_s} \alpha_1(h_b \Lambda) = e^{\beta_s(\La)} ,$$ and $$\sigma(\La) : = \limsup_{s \rightarrow \infty} \frac{\sigma_s(\La)}{s^{\frac{n}{d}}}.$$ 
\noindent It is important to note that $\sigma(\La)$ does depend on our choice of norm. As in \S\ref{sec:markoff:minkowski}, given discrete $\Omega \in \R^2$ without vertical vectors, denote the $(m,n)$-Markoff constant by $\tilde{c} (\Omega)$, and define the $(m,n)$-\emph{Markoff constant} of $\La$ by $\tilde{c} (\La) = \tilde{c}(\Pi_{m,n}(\La))$. This is the infimum of the set of $\tilde{c} >0$ so that there exist a sequence $\{\vv_k\} \subset \Lambda$ with $$\|p_1(\vv_k)\| \le \tilde{c} \|p_2(\vv_k)\|^{-\frac{n}{m}}.$$

\noindent Now suppose $\La$ satisfies the Minkowski property, that is, there is an upper bound on the volumes of convex centrally symmetric sets in $\R^d$ that do not intersect $\La$. Let $c(\La)$, the \emph{Minkowski constant} of $\La$, be the supremum of the volumes of such sets. As in \S\ref{sec:markoff:minkowski}, let $a_j$ denote the volume of the unit (Euclidean norm) ball in $\R^j$.

\begin{Prop}\label{prop:abstract:markoff} We have \begin{equation}\label{eq:markoff:minkowski} c(\La) \geq \tilde{c}(\La)^m a_m a_n.\end{equation} Furthermore, we have \begin{equation}\label{eq:sigma:markoff}\sigma(\La) = \tilde{c}(\La)^{-\frac{m}{d}}.\end{equation} When $\tilde{c}(\La) = 0$ this is taken to mean $\sigma(\La) = \infty.$
\end{Prop}

\medskip

\noindent\textbf{Proof of Proposition~\ref{prop:markoff:mink}:} Let $X$ be a Mahler-Minkowski system. Then applying Proposition~\ref{prop:abstract:markoff} to $\La_x$ yields Proposition~\ref{prop:markoff:mink}.\qed

\section{Hyperbolic surfaces}\label{sec:hyp} As discussed in \S\ref{subsec:dev}, we will describe how to associate a discrete subset in $\R^2\minuszero$ to each cusp of $X = SL(2, \R)/\Gamma$, where $\Gamma$ is a non-uniform lattice in $SL(2, \R)$. Let $\rho$ denote the contragredient representation of $SL(2, \R)$ on $\R^2$, that is, for $\vv \in \R^2$, $\rho(g) \vv = (g^{-1})^T\vv$. Let $\Delta \subset \Gamma$ be a maximal parabolic subgroup, and let $\vv_0 \in \R^2\minuszero$ be such that $\rho(\Delta) \vv_0  = \vv_0$. Let $X_{\Delta} \subset X$ denote the cusp corresponding to $\Delta$. Fix $x_0 \in X$, and given $x \in X$, define
\begin{equation}\label{eq:delta:dist:def}
d_{\Delta}(x, x_0) =
\left\{ \begin{array}{ll}  d(x, x_0) & x \in X_{\Delta} \\ 0 & \mbox{otherwise}\end{array}\right.
\end{equation}
\noindent and, writing $x = g\Gamma$, define $\La_{x, \Delta} := \rho(g\Gamma)\vv_0 = \{ \rho(g\gamma)\vv_0: \gamma \in \Gamma\}$. Our main lemma is a `quantitative Mahler' condition for $\La_{x, \Delta}$:

\begin{lemma}\label{lemma:quant:hyp:mahler} Fix notation as above. Let $\{x_n\}_{n = 1}^{\infty} \subset X$ be such that $$d_{\Delta}(x_n, x_0) \rightarrow \infty.$$ Then \begin{equation}\label{eq:quant:hyp:mahler} \lim_{n \rightarrow \infty} \frac{d_{\Delta}(x_n, x_0)}{2 \log \alpha_1(\La_{x_n, \Delta})} = 1.\end{equation}
\end{lemma}

\begin{proof} Conjugating if necessary, we assume $$\Delta = \left\{ \left(\begin{array}{cc}1 & m \\0 & 1\end{array}\right):  m \in \Z \right\}.$$ Let $\vv_0 : = \left(\begin{array}{c}1 \\0\end{array}\right)$, so $\rho(\Delta) \vv_0 = \vv_0$. Note that if we use the Euclidean norm on $\R^2\minuszero$, $\alpha_1(\rho(g\Gamma)\vv_0)$ is an $K = SO(2)$-invariant function on $X$, since $\rho(K)$ acts via isometries on $\R^2$. $\h^2/\Gamma$ = $K \backslash X$ is foliated by translates of the closed horocycle orbit corresponding to $\delta$. To understand how far $x = g\Gamma$ is into the cusp corresponding to $\Delta$, we calculate the length of the closed horocycle it (or a $K$-translate of it) is on. 

Let $x_0$ be the identity coset, i.e. $x_0 = e \Gamma$, where $e \in SL(2, \R)$ is the identity element. By our construction of $\Delta$, $\{g_{-t} x_0\}_{t \geq 0}$ is divergent in $X$ (going into the cusp corresponding to $\Delta$), with $d(g_{-t} x_0, x_0) = t$, and so $d_{\Delta}(g_{-t} x_0, x_0) = t$. Note that the length of the closed horocycle that $g_{-t} x_0$ is on is $e^{-t}$, since $$g_{-t} \left(\begin{array}{cc}1 & 1 \\0 & 1\end{array}\right) g_t =  \left(\begin{array}{cc}1 & e^{-t} \\0 & 1\end{array}\right),$$ and the shortest vector in $\rho(g_{-t}\Gamma) \vv_0$ is given by $\rho(g_{-t} )\vv_0$, so $$\alpha_1(\rho(g_{-t}\Gamma)\vv_0) = e^{\frac{t}{2}}.$$ Write $x_n = k_n p_n \Gamma$, where $k_n \in K$, and $p_n$ is upper triangular. Since both $d$ and $\alpha_1$ are $K$-invariant, we can assume $k_n = e$, so $x_n = p_n \Gamma$. Since $d_{\Delta}(x_n, x_0) \rightarrow \infty$, we can write $$p_n =  \left(\begin{array}{cc}e^{-\frac{t_n}{2}} & b \\0 & e^{\frac{t_n}{2}}\end{array}\right),$$ with $t_n \rightarrow \infty$. We have that, as $n \rightarrow \infty$, $d_{\Delta}(x_n, x_0) \sim t_n$, and $\alpha_1(\rho(p_n\Gamma)\vv_0) = e^{\frac{t_n}{2}}$.
\end{proof}

\medskip

\noindent Now let $\Delta_1, \ldots \Delta_k$ denote the conjugacy classes of maximal parabolic subgroups of $\Gamma$. Given $I \subset \{1, \ldots, k\}$, and $x, y \in X$, define $$d_{I} (x, y) = \max_{i \in I} d_{\Delta_i}(x, y).$$ $d_{I}$ measures distance into the cusp(s) corresponding to $I$. For $j \in \{1, \ldots, k\}$, let $\vv_j \in \R^2\minuszero$ be such that $\rho(\Delta_j) \vv_j = \vv_j$, and for $x \in X$, let $$\La_{x, I} : = \bigcup_{i \in I} \rho(g\Gamma) \vv_i.$$ Lemma~\ref{lemma:quant:hyp:mahler} yields the following:

\begin{Cor}\label{cor:quant:hyp:mahler} Fix $x_0 \in X$. Let $\{x_n\}_{n = 1}^{\infty} \subset X$ be such that $$d_{I}(x_n, x_0) \rightarrow \infty.$$ Then \begin{equation}\label{eq:quant:hyp:I:mahler} \lim_{n \rightarrow \infty} \frac{d_{I}(x_n, x_0)}{2 \log \alpha_1(\La_{x_n, I})} = 1.\end{equation}
\end{Cor}

\subsection{Proof of Theorem~\ref{theorem:hyp:excursions}}\label{subsec:hyp:excursions} As discussed in \S\ref{subsubsec:hyp} we will in fact prove a more general statement about excursions. As above, let $I$ denote a subset of the cusps of $X$. We have

\begin{Claim}\label{claim:hyp:excursions}Fix $x_0 \in X$, and suppose $x$ is not cuspidal. Then 
$$\limsup_{|s| \rightarrow \infty} \frac{d_I(h_s x, x_0)}{\log |s|} = 1 + \limsup_{t \rightarrow \infty} \frac{d_I(g_t x, x_0)}{t} $$\end{Claim}
\begin{proof} Let $$\beta(x) = \limsup_{|s| \rightarrow \infty} \frac{\log\alpha_1(\La_{h_s x, I})}{\log |s|}$$ and $$\gamma(x)  = \limsup_{t \rightarrow \infty} \frac{\log\alpha_1(\La_{g_t x, I})}{t}.$$ We would like to use Theorem~\ref{theorem:axiomatic} to conclude that $\beta(x) = \frac{1}{2} + \gamma(x)$, which, combined with Corollary~\ref{cor:quant:hyp:mahler} would give us our result. However, $\La_{x, I}$ satisfies the Minkowski condition only in the case when $\Gamma$ has one cusp (see \S\ref{sec:hyp:mink}). Thus, we will use Theorem~\ref{theorem:abstract}, and the geometric observation that for any $y \in X$, $d(h_s y, y) \le 2\log s$, and $d(g_t y, y) \le t$, and the fact that Corollary~\ref{cor:quant:hyp:mahler} guarantees that $\gamma(x) \in [0, \frac{1}{2}]$. Thus, by part (4) of Theorem~\ref{theorem:abstract}, $(0, 2) \subset \Exp(\La_{x, I})$, and so by part (1), $\beta \geq \frac{1}{2}$, allowing us to apply part (2). Combining these results, we obtain, as in the conclusion of Theorem~\ref{theorem:axiomatic}, $$\beta(x) =  \frac{1}{2} + \gamma(x)  =  1- \frac{1}{\mu(\La_{x, I})},$$ which, applying Corollary~\ref{cor:quant:hyp:mahler}, yields our result. Finally, note that the case $I = \{1, \ldots, k\}$ yields Theorem~\ref{theorem:hyp:excursions}. 
\end{proof}
\medskip

\subsubsection{One-cusped surfaces}\label{sec:hyp:mink} If $I = \{ 1,  \dots, k\}$, the full set of cusps, we can directly apply Theorem~\ref{theorem:axiomatic}. We write $\La_x$ to denote $\La_{x, \{1, 2, \ldots, k\}}.$ We claim for all $x \in X$, there is a $C >0$ such that $\La_x \cap M$ is non-empty for all convex, centrally symmetric $M$ of volume at least $C$. Using the transitivity of the $SL(2, \R)$-action and Lemma~\ref{lemma:john}, it suffices to show:

\begin{Prop}\label{prop:hyp:mink} There is a constant $R_0 >0$ so that for any $x \in X$, $$\La_{x} \cap B(0, R_0) \neq \emptyset.$$\end{Prop} 

\begin{proof} The proof of Lemma~\ref{lemma:quant:hyp:mahler} shows that the shortest vector in $\La_{x, I}$ corresponds to the shortest closed horocycle (based at the cusps corresponding to $I$) that the point $Kg\Gamma \in \h^2/\Gamma$ is on. If $I = \{1, \ldots, k\}$, this is bounded above, since the complement of the cusps is compact.

We also note (as C.~Judge pointed out to us) that if $I$ is a proper subset of $\{1, \ldots, k\}$, this statement is false. For a point deep in a cusp in $I^c$, the shortest closed horocycle from $I$ that $x$ is on (which will be the shortest vector in $\La_{x, I}$) can be made arbitrarily large.
\end{proof}

\medskip

\section{Diophantine exponents of discrete sets}\label{sec:diophexp}

In this section we prove our main abstract results Theorem~\ref{theorem:abstract} and Proposition~\ref{prop:abstract:markoff}. We first recall, in \S\ref{subsec:diag:flows2}, a general lemma from~\cite{Kleinbock:GAFA} on diagonal flows on $\R^2$, and then prove the main theorem (\S\ref{subsec:proof:theorem:abstract}. The strategy in the proofs is to use approximates to find orbit points at which we the associated discrete set has a short vector of an appropriate length, and vice versa, when we know there are specified orbit points with short vectors, to derive the existence of a sequence of approximates. We use a similar strategy, with more delicate estimates, along with some basic convex geometry, to prove Proposition~\ref{prop:abstract:markoff} in \S\ref{subsec:abstract:markoff:proof}.

\subsection{Diagonal flows on $\R^2$}\label{subsec:diag:flows2} We recall a general lemma (\cite[Lemma 2.1]{Kleinbock:GAFA}) on diagonal flows on $\R^2$ that will be used in the proof of both Proposition~\ref{prop:dioph:2} and Theorem~\ref{theorem:abstract}. We fix notation: let $\Omega \subset \R^2\minuszero$ be discrete and without vertical vectors, and let $m, n \in \N$, and $d = m +n$. Let  
\begin{equation}\label{eq:at}
a_t : = \left(\begin{array}{cc}e^{\frac{n}{d} t} & 0 \\0 & e^{-\frac{m}{d} t} \end{array}\right)
\end{equation}
\noindent Note that $\{a_t\}_{ t \in \R}$ is a subgroup of $GL(2, \R)$. Let 
\begin{equation}\label{eq:gamma:mn}\gamma_{m,n}(\Omega) = \limsup_{t \rightarrow \infty} \frac{\log(\alpha_1(a_t \Omega))}{t}\end{equation}

\begin{lemma}\label{claim:prop:3} $$\gamma_{m,n}(\Omega) = \frac{m}{d} -\frac{1}{\mu(\Omega)}.$$\end{lemma}
\noindent 
\medskip To keep the paper as self-contained as possible, and to avoid difficulties in notation, we recall the proof:

\begin{proof} We first show for any $\nu \in \Exp(\Omega)$,  $$\gamma_{m,n}(\Omega) \geq \frac{m}{d} -\frac{1}{\mu(\Omega)}.$$ Let $\nu \in \Exp(\Omega)$, $\nu >0$. Let $c>0$, $(x_k, y_k)^T \in \Omega$ be such that $\frac{|x_k|}{|y_k|} \rightarrow \infty$, and 
\begin{equation}\label{eq:xy:geod:nu} |x_k| \le c |y_k|^{-(\nu -1)}.\end{equation} 

\noindent Let $t_k = \log \frac{|x_k|}{|y_k|}$. Then $$e^{\frac{n}{d}t_k} |x_k| = e^{-\frac{m}{d}t_k} |y_k|.$$

\noindent  Using (\ref{eq:xy:geod:nu}), we have $$e^{t_k} =  \frac{|x_k|}{|y_k|} > \frac{1}{c} |y_k|^{\nu}.$$\noindent Taking logarithms, and reorganizing terms, we obtain
\begin{equation}\label{eq:logy:nu} -\log |y_k| > -\frac{t_k}{\nu} - \frac{\log c}{\nu}\end{equation}

\noindent $\alpha_1(a_{t_k} \Omega) > e^{\frac{m}{d}t_k} |y_k|^{-1}$, and taking logarithms, we get

\begin{equation}\label{eq:gt:alpha1:nu} \log \alpha_1(a_{t_k} x) > \frac{m}{d}t_k - \log |y_k| > \left(\frac{m}{d} - \frac{1}{\nu}\right ) t_k - \frac{\log c}{\nu}\end{equation}

\noindent which yields our inequality. To finish the proof, we need:
\begin{claim} For all $\eta \in  \left(0, \frac{1}{\frac{m}{d} - \gamma}\right)$, $\eta \in \Exp(\Omega)$. Thus, $\gamma(\Omega) \le \frac{m}{d} -\frac{1}{\mu(\Omega)}$.\end{claim}
\medskip
By construction, $0 \le \gamma(\Omega) \le \frac{m}{d}$, so the interval $\left(0, \frac{1}{\frac{m}{d} - \gamma_{m,n}(\Omega)}\right)$ is non-empty. Let $\eta \in \left(0, \frac{1}{\frac{m}{d} - \gamma(\Omega)}\right)$, so $$\gamma_{m,n}(\Omega) \geq \frac{m}{d} - \frac{1}{\eta}.$$

\noindent Thus, there is a sequence $t_k \rightarrow \infty$, and $(x_k, y_k)^T \in \Omega$ such that 

\begin{eqnarray}\label{eq:gt:tk} e^{-\frac{m}{d}t_k} |y_k| &\le& e^{-\left(\frac{m}{d} - \frac{1}{\eta}\right)t_k} \\ \nonumber
e^{\frac{n}{d} t_k} |x_k| &\le& e^{-\left(\frac{m}{d} - \frac{1}{\eta}\right)t_k} \end{eqnarray}

\noindent Thus, $|y_k| \le e^{\frac{1}{\eta} t_k}$ and $|x_k| \le e^{t_k\left(\frac{1}{\eta} -1\right)}$. Since $\eta >0$, we have $|y_k|^{-\eta} > e^{-t_k}$, so 
\begin{equation}\label{eq:x:y:at} |x_k| < e^{-t_k\left(1-\frac{1}{\eta} \right)} < |y_k|^{-\eta\left(1-\frac{1}{\eta} \right)} = |y_k|^{-(\eta -1)}\end{equation}
Thus, we obtain a sequence $(x_k, y_k)^T \in \La_x$ so that 
$$|x_k| \le |y_k|^{-(\eta -1)}$$
\noindent We need to show that $\frac{|y_k|}{|x_k|} \rightarrow \infty$ (at least along some subsequence). Suppose not. Then there is a $D> 0$ so that $\frac{|y_k|}{|x_k|}  \le D$, i.e. $|y_k| \le D|x_k|$. (\ref{eq:x:y:at}) yields $$|y_k| \le D|x_k| \le D|y_k|^{-(\eta -1)}$$\noindent Thus, $(x_k, y_k) \rightarrow (0, 0)$, which is a contradiction to discreteness, proving the claim, and thus the lemma.\end{proof}

\subsection{Proof of Theorem~\ref{theorem:abstract}}\label{subsec:proof:theorem:abstract} We will first prove claims, corresponding to the first two parts of Theorem~\ref{theorem:abstract}, and then use Lemma~\ref{claim:prop:3} to prove the last two parts. Fix notation as in \S\ref{subsec:notation}: $\La \subset \R^d\minuszero$ discrete, $\beta = \beta(\La)$ and $\gamma = \gamma (\La)$ where $\beta(\La)$ and $\gamma(\La)$ are as in (\ref{eq:beta:la:def}). The symbol $\| \cdot\|$ will be used to denote $L^2$-norm on $\R^m, \R^n$, and $M_{n \times m}(\R)$, and the $(m,n)$-norm on $\R^d$ given by $\|\vv\| = \max(\|p_1(\vv)\|, \|p_2(\vv)\|)$.  
\begin{Claim}\label{claim:theorem:1} For all $\nu \in \Exp(\La)$, there is a $c>0$ so that \begin{equation}\label{eq:horod:lb} \limsup_{s \rightarrow \infty} \frac{\sup_{h \in \B_s} \alpha_1(h\La)}{|s|^{1-\frac{1}{\nu}}} \geq c.\end{equation} Thus $\beta = \beta(\La) = \limsup_{s \rightarrow \infty} \frac{\sup_{ h \in \B_s} \log \alpha_1(h \Lambda)}{\log s}
\geq 1-\frac{1}{\nu}$.
\end{Claim}

\begin{proof} Let $\nu \in \Exp(\La)$. Thus, there is a $c_0>0$ and a sequence $\vv_k \in \La$ so that $\frac{\|p_2(\vv_k)\|}{\|p_1(\vv_k)\|} \rightarrow \infty$ and
\begin{equation}\label{eq:nu:d} \|p_1(\vv_k)\| \le c_0  \|p_2(\vv_k)\|^{-(\nu-1)}\end{equation}

\noindent Write $p_2(\vv_k) = \left(\begin{array}{c}y_k^{(1)} \\\vdots \\y_k^{(n)}\end{array}\right)$. For $1 \le i \le n$, let \begin{equation}\label{eq:row:def} r_k^{(i)} = -\frac{y_k^{(i)}}{\|p_1(\vv_k)\|^2} p_1(\vv_k).\end{equation} We have $r_k^{(i)} \cdot p_1(\vv_k) = -y_k^{(i)}$ and $\|r_k^{(i)}\| = -\frac{|y_k^{(i)}|}{\|p_1(\vv_k)\|}$. Let $B_k$ be the $n \times m$ matrix whose $i^{th}$  row is $r_k^{(i)}$ (transposed). Note that  $\|B_k\| = s_k = \frac{\|p_2(\vv_k)\|}{\|p_1(\vv_k)\|}$. Note that for any $B \in M_{n \times m}(\R)$, and $\vv \in \R^d$, 
\begin{eqnarray}\label{eq:B:v} p_1(h_B \vv) &=& p_1(\vv)\\ \nonumber p_2(h_B \vv) &=& Bp_1(\vv) + p_2(\vv), \end{eqnarray} 

\noindent Since $B_k p_1(\vv_k) = - p_2(\vv_k)$ by construction, we have
\begin{eqnarray}\label{eq:B:k:v} p_1(h_{B_k} \vv_k) &=& p_1(\vv_k)\\ \nonumber p_2(h_{B_k} \vv_k) &=& 0, \end{eqnarray} 

\noindent Thus, 
\begin{equation}\label{eq:alpha1:d} \alpha_1(h_{B_k} \La)  \geq \frac{1}{\|p_1(\vv_k)\|}.\end{equation}

\noindent Rewriting (\ref{eq:nu:d}), we obtain
$$1 \geq c_0^{-1} \|p_1(\vv_k)\| \|p_2(\vv_k)\|^{\nu-1}$$

\noindent and raising both sides to the power $1/\nu$, we have
$$1 \geq c_0^{-1/\nu}\|p_1(\vv_k)\|^{\frac{1}{\nu}} \|p_2(\vv_k)\|^{1-\frac{1}{\nu}}$$

\noindent Replacing the $1$ in (\ref{eq:alpha1:d}) with the above expression, we obtain
\begin{equation}\label{eq:alpha1:nu:d} \alpha_1(h_{B_k} x)  \geq c_0^{-\frac{1}{\nu}}\|B_k\|^{1-\frac{1}{\nu}}.
\end{equation}

\noindent Setting $c = c_0^{-\frac{1}{\nu}}$, we obtain (\ref{eq:horod:lb}).

\end{proof}

\begin{Claim}\label{claim:theorem:2} Suppose $\beta = \beta(\La)> 0$. Then  $\left(1, \frac{1}{1-\beta}\right) \subset \Exp (\La)$.
\end{Claim}

\begin{proof} Let $\eta \in \left(1, \frac{1}{1-\beta}\right)$, and let $\nu \in \left(\eta, \frac{1}{1-\beta}\right)$, so 
$$\beta > 1- \frac{1}{\nu} > 1- \frac{1}{\eta},$$ and there is a sequence $s_k \rightarrow \infty$, $\vv_k \in \La$, and $B_k \in M_{n\times m}(\R)$ with $\|B_k\| \le s_k$ so that 
$$\|h_{B_k} \vv_k\| \le s_k^{-(1-\frac{1}{\nu})}.$$ Thus, \begin{eqnarray}\label{eq:k:d:nu} \|p_1(\vv_k)\| &\le& s_k^{-(1-\frac{1}{\nu})}\\  \|B_kp_1(\vv_k) + p_2(\vv_k)\| &\le& s_k^{-(1-\frac{1}{\nu})}\nonumber \end{eqnarray}
\noindent Rewriting the second equation, we obtain
$$\|p_2(\vv_k)\| \le \|B_k p_1(\vv_k)\| + s_k^{-(1-\frac{1}{\nu})} \le \|B_k\| \|p_1(\vv_k)\| + s_k^{-(1-\frac{1}{\nu})}.$$ Since $\|B_k\| \le s_k$, we have $\|p_2(\vv_k)\| \le s_k^{\frac{1}{\nu}} + s_k^{-(1-\frac{1}{\nu})}$. Now, \begin{eqnarray}\label{eqnarray:k}\nonumber \|p_1(\vv_k)\| \le s_k^{-(1-\frac{1}{\nu})} & =& \|p_2(\vv_k)\|^{(1-\eta)}\|p_2(\vv_k)\|^{(\eta -1)}s_k^{-(1-\frac{1}{\nu})} \\ \nonumber &\le&  \|p_2(\vv_k)\|^{(1-\eta)}\left(s_k^{\frac{1}{\nu}} + s_k^{-(1-\frac{1}{\nu})}\right)^{\eta-1}s_k^{-(1-\frac{1}{\nu})} \\  &\le &  \|p_2(\vv_k)\|^{(1-\eta)} 2^{\eta-1} s_k^{\frac{\eta-1}{\nu}} s_k^{-(1-\frac{1}{\nu})} \end{eqnarray}

\noindent where in the last line we are using $s_k^{\frac{1}{\nu}} > s_k^{\frac{1}{\nu} -1}$ (since we can assume $s_k >1$). Combining the powers in the last line of (\ref{eqnarray:k}), we obtain $s_k^{\frac{\eta}{\nu} -1}$. Since $\frac{\eta}{\nu} -1 < 0$, we can define $C = \max_{k} 2^{\eta-1} s_k^{\frac{\eta}{\nu} -1} < \infty$,  and (\ref{eqnarray:k}) yields
$$\|p_1(\vv_k)\| \le C \|p_2(\vv_k)\|^{-(\eta-1)}$$
\noindent Finally, the discreteness of $\La$ implies that (along a subsequence) we must have $\frac{|p_2(\vv_k)\|}{\|p_1(\vv_k)\|} \rightarrow \infty$, otherwise, if there was a bound $D \le \infty$ so that $\|p_2(\vv_k)\| \le D \|p_1(\vv_k)\|$, we would have $\|p_2(\vv_k)\| \le (DC)^{\frac{1}{\eta}}$ and $\|p_1(\vv_k)\| \le DC^{1-\frac{1}{\eta}}$, so that $\vv_k$ would have to cycle through a finite set, a contradiction to $\beta >0$.

\end{proof}

\begin{Claim}\label{claim:discrete} Suppose $\La \subset \R^d$ is discrete. then $\Pi_{m,n}(\La)$ is discrete in $\R^2$.
\end{Claim}

\begin{proof} Recall that a subset $A$ of Euclidean space is discrete if and only if for any $R >0$, $\{ \vv \in A: \| \vv\| \le R\}$ is finite, and that this statement is independent of the norm. Let $\| \cdot \|_{\sup}$ denote $\sup$-norm on $\R^2$, and $\| \cdot \|$ denote our our usual norm on $\R^d$. Then we have, for all $\vv \in \R^d$, $$\|\Pi_{m,n} (\vv)\|_{\sup} = \|\vv\|.$$ Thus, since $\La$ is discrete, $\{ \vv \in \Pi_{m,n} (\La): \|\vv\| \le R\}$ is finite for all $R >0$, and so $\Pi_{m,n}(\La)$ is discrete.

\end{proof}

\noindent To complete the proof of Theorem~\ref{theorem:abstract}, apply Lemma~\ref{claim:prop:3} to the set $\Omega = \Pi_{m,n}(\La)$, observing that $$\Pi_{m,n}(g_t \vv) = a_t \Pi_{m,n}(\vv)$$ for all $\vv \in \R^d$. Since $\La$ is assumed to be discrete, $\Omega = \Pi_{m,n}(\La)$ is discrete as well.\qed

\subsection{Proof of Proposition~\ref{prop:abstract:markoff}}\label{subsec:abstract:markoff:proof} We divide the proof into two parts. In \S\ref{subsec:markoff:convex}, we prove (\ref{eq:markoff:minkowski}) and in \S\ref{subsec:markoff:approx} we prove (\ref{eq:sigma:markoff}).

\subsubsection{Markoff constants and convex sets}\label{subsec:markoff:convex} We follow the proof of Lemma~\ref{lemma:exp:2}. Let $c = c(\La)$ denote the Minkowski constant of $\La$, and let $c^{\prime} > \left(\frac{c}{a_m a_n}\right)^{\frac{1}{m}}$. For $j \in \N$,

$$K_j : = \Pi_{m,n}^{-1}\left(\left\{\left(\begin{array}{c}x \\y\end{array}\right) \in \R^2: |x| \le \frac{c^{\prime}}{j}, |y| \le j^{\frac{m}{n}}\right\}\right).$$

\noindent $K_j$ is convex, centrally symmetric, and has volume $\geq c$, so we can find 
$$\vv_j \in \La \cap K_j.$$

\noindent By the definition of $K_j$, 
\begin{equation}\label{eq:c:cprime} \|p_1(\vv_j)\| \le \frac{c^{\prime}}{j} \le c^{\prime}\|p_2(\vv_j)\|^{-\frac{n}{m}} = c^{\prime}\|p_2(\vv_j)\|^{-(\frac{d}{m}-1)}.\end{equation}

\noindent As above, if $\|p_2(\vv_j)\|/\|p_1(\vv_j)\|$ were bounded, we would have $\vv_j \rightarrow 0$, a contradiction to discreteness. Thus, (\ref{eq:c:cprime}) yields $c^{\prime} \geq \tilde{c}(\La)$, and since $c^{\prime}$ was arbitrary, we have 
$$\left(\frac{c}{a_m a_n}\right)^{\frac{1}{m}} \geq \tilde{c}(\La).$$
\noindent Rewriting, we obtain (\ref{eq:markoff:minkowski}). \qed
\subsubsection{Markoff constants and approximates}\label{subsec:markoff:approx} We divide the proof of (\ref{eq:sigma:markoff}) into two claims:

\begin{Claim}\label{claim:sigma:markoff:1} Let $c > \tilde{c}(\Lambda)$. Then $\sigma(\La) > c^{-\frac{m}{d}}$. Thus $\sigma(\La) \geq \tilde{c}(\La)^{-\frac{m}{d}}$. 
\end{Claim}

\begin{proof} For $c > \tilde{c}(\La)$, there is a sequence $\vv_k \in \La$ such that $$\|p_1(\vv_k)\| \le c \|p_2(\vv_k)\|^{-\frac{n}{m}},$$ which we rewrite as $$1 \geq \left(\|p_1(\vv_k)\| \|p_2(\vv_k)\|^{\frac{n}{m}} c^{-1}\right)^{\frac{m}{d}}.$$As in the proof of Claim~\ref{claim:theorem:1}, let $B_k \in M_{n \times m}(\R)$ be the matrix whose rows are given by $r_k^{(i)}$, defined by (\ref{eq:row:def}). Thus $h_{B_k} \vv_k$ satisfies $$p_1(h_{B_k}\vv_k) = p_1(\vv_k) \mbox{ and } p_2(h_{B_k} \vv_k) = 0,$$ and $s_k : = \|B_k\| = \frac{\|p_2(\vv_k)\|}{\|p_1(\vv_k)\|}$.  Thus 
\begin{eqnarray}\nonumber \sigma_{s_k}(\La)  \geq \alpha_1(h_{B_k}\La) &\geq& \frac{1}{\|p_1(\vv_k)\|}\\ \nonumber &\geq& \frac{\left(\|p_1(\vv_k)\| \|p_2(\vv_k)\|^{\frac{n}{m}} c^{-1}\right)^{\frac{m}{d}}}{\|p_1(\vv_k)\|} \\ &\geq& c^{-\frac{m}{d}} s_k^{\frac{n}{d}}\end{eqnarray}
Thus $\sigma(\La) > c^{-\frac{m}{d}}$ as desired.

\end{proof}
\medskip

\noindent To complete the proof of (\ref{eq:sigma:markoff}) and thus Proposition~\ref{prop:abstract:markoff}, we require the following

\begin{Claim} Let $c^{\prime} > (\sigma(\La))^{-\frac{d}{m}}$. Then there are infinitely many $\vv \in \La$ so that $$\|p_1(\vv)\| \le c^{\prime} \|p_2(\vv)\|^{-\frac{n}{m}}.$$ Thus, $\sigma(\La) \le \tilde{c}(\La)^{-\frac{m}{d}}$. 

\end{Claim}

\begin{proof} Let $c^{\prime}> c> \sigma(\La)^{-\frac{d}{m}}$. Then $c^{-\frac{m}{d}} < \sigma(\La)$, so $\exists s_k \rightarrow \infty$, $B_k \in M_{n \times m}(\R)$ with $\|B_k\| \le s_k$, and $\vv_k \in \La$ so that $\|h_{B_k}\vv_k\| \le c^{\frac{m}{d}}s_k^{-\frac{n}{d}}$. Again using (\ref{eq:B:v}), we can rewrite this as
\begin{equation}\label{eq:Bk:sk:m}
\|p_1(\vv_k)\| \le c^{\frac{m}{d}}s_k^{-\frac{n}{d}}
\end{equation}
\begin{equation}\label{eq:Bk:sk:n} 
\|B_k p_1(\vv_k) + p_2(\vv_k)\| \le c^{\frac{m}{d}}s_k^{-\frac{n}{d}}
\end{equation}
\noindent We will show that for $k >>0$, 
\begin{equation}\label{eq:xk:yk:c}
\|p_1(\vv_k)\| \le c^{\prime} \|p_2(\vv_k)\|^{-\frac{n}{m}}
\end{equation}
\noindent (\ref{eq:Bk:sk:m}) yields
$$\|p_1(\vv_k)\| \le c^{\prime} \|p_2(\vv_k)\|^{-\frac{n}{m}} \left(c^{\prime}\right)^{-1} \|p_2(\vv_k)\|^{\frac{n}{m}} c^{\frac{m}{d}} s_k^{-\frac{n}{d}} $$
\noindent Thus, it suffices to show that, for $k >>0$, 
$$\left(c^{\prime}\right)^{-1} \|p_2(\vv_k)\|^{\frac{n}{m}} c^{\frac{m}{d}} s_k^{-\frac{n}{d}} \le 1,$$ i.e., 
\begin{equation}\label{eq:yk:c:cprime}
\|p_2(\vv_k)\|^{\frac{n}{m}} \le c^{\prime} c^{-\frac{m}{d}} s_k^{\frac{n}{d}}
\end{equation}
\noindent Using (\ref{eq:Bk:sk:n}), we have 
$$\|p_2(\vv_k)\| \le \|B_k p_1(\vv_k)\| + c^{\frac{m}{d}} s_k^{-\frac{n}{d}}$$
\noindent Since $\|B_k p_1(\vv_k)\| \le \|B_k\| \|p_1(\vv_k)\|$, we can rewrite this as 
\begin{eqnarray}\label{eq:vk:m:c}\nonumber \|p_1(\vv_k)\| &\le& s_k c^{\frac{m}{d}} s_k^{-\frac{n}{d}} + c^{\frac{m}{d}} s_k^{-\frac{n}{d}} \\ &=& c^{\frac{m}{d}} s_k^{\frac{m}{d}}\left(1 + s_k^{-1}\right)
\end{eqnarray}
\noindent For $k >>0$, $\left(1 + s_k^{-1}\right) < \left(\frac{c^{\prime}}{c}\right)^{\frac{m}{n}}$, so 
$$\left(1 + s_k^{-1}\right)^{\frac{n}{m} }< \frac{c^{\prime}}{c},$$ so 
\begin{equation}\label{eq:c:n:d}c^{\frac{n}{d}} \left(1 + s_k^{-1}\right)^{\frac{n}{m}} < c^{\prime} c^{-\frac{m}{d}}\end{equation}
\noindent Raising (\ref{eq:vk:m:c}) to the $\frac{n}{m}$ power, and combining with (\ref{eq:c:n:d}), we obtain (\ref{eq:yk:c:cprime}), as desired.
\end{proof}

\subsubsection{One-sided results}\label{sec:proof:prop:mark:2} Given $\vv = \left(\begin{array}{c}\vv_1 \\ \vv_2\end{array}\right) \in \R^2$, we define 
\begin{equation}\label{eq:signum:def} \sgn(\vv) : = \sgn\left(\frac{\vv_1}{\vv_2}\right).\end{equation} Given a discrete set in $\La \in \R^2$, let \begin{eqnarray*} \La^+ &=& \{\vv \in \La: \sgn(\vv) = 1\} \\ \La^- &=& \{\vv \in \La: \sgn(\vv) = -1\}\\ \end{eqnarray*}. Let $X$ be a $SL(2,\R)$-Minkowski system (for example $X_2$), and for $x \in X$,  let \begin{eqnarray*} \mu^+(x) &=& \sup \Exp(\La_x^+)\\  \mu^- (x) &=& \sup \Exp(\La_x^-) \\ \tilde{c}^+(x) &=& \tilde{c}(\La_x^+) \\\tilde{c}^-(x) &=& \tilde{c}(\La_x^-)\\ \end{eqnarray*}

\begin{Prop}\label{prop:one:sided} Suppose $\La_x$ does not have vertical vectors. Then 
\begin{equation}\label{eq:dioph:plus} \limsup_{s \rightarrow +\infty} \frac{ \log \alpha_1(h_s x)}{\log s} = 1-\frac{1}{\mu^{-}(x)}\end{equation}
\begin{equation}\label{eq:dioph:minus} \limsup_{s \rightarrow -\infty} \frac{ \log \alpha_1(h_s x)}{\log |s|} = 1-\frac{1}{\mu^{+}(x)}\end{equation}
\begin{equation}\label{eq:markoff:plus} \limsup_{s \rightarrow +\infty} \frac{\alpha_1(h_s x)}{s^{\frac{1}{2}}} = \sigma^{-}(x)^{-\frac{1}{2}}\end{equation}
\begin{equation}\label{eq:markoff:minus} \limsup_{s \rightarrow -\infty} \frac{\alpha_1(h_s x)}{|s|^{\frac{1}{2}}} = \sigma^{+}(x)^{-\frac{1}{2}}\end{equation}

\end{Prop}

\begin{proof} (\ref{eq:dioph:plus}) and (\ref{eq:dioph:minus}) are straightforward generalizations of~\cite[Proposition 3.3]{AM1}. There is a typo in the statement there, the limits are given as $\mu^{\pm}$ when they should be $1- \frac{1}{\mu^{\pm}}$. We simply observe that in the proof of Proposition~\ref{prop:dioph:2}, the time $s$ corresponding to an approximate $\vv$ is given by $s = - \frac{\vv_1}{\vv_2}$. For (\ref{eq:markoff:plus}) and (\ref{eq:markoff:minus}), the same observation applied to the proof of Proposition~\ref{prop:abstract:markoff} in the case $m = n = 1$ yields the result.
\end{proof}

\section{Deviations}\label{sec:devproof} In this section, we prove Theorem~\ref{theorem:hyp:deviations}, assuming Theorem~\ref{theorem:hyp:excursions}. The basic idea of the proof is that if a horocycle trajectory is deep into the cusp, it cannot return to the compact part of the space too quickly, and thus, it must spend a definite proportion of time \emph{outside} the compact part. This bears some spiritual similarity to the arguments of \cite[\S4]{Strombergsson}. 

Fix notation as in \S\ref{subsubsec:hyp}, and fix $x, x_0 \in X = SL(2, \R)/\Gamma$, and assume $x$ is not cuspidal, and $$\gamma = \gamma(x) =  \limsup_{t \rightarrow \infty} \frac{d(g_t x, x_0)}{t} >0.$$ \noindent By construction, $\gamma \in [0, 1]$. Theorem~\ref{theorem:hyp:excursions} then states that \begin{equation}\label{eq:beta:hyp}\beta = \beta(x) = \limsup_{|s| \rightarrow \infty} \frac{d(h_s x, x_0)}{\log|s|} = 1+ \gamma >1.\end{equation} Fix $\epsilon \in (0, \gamma)$ and $\delta \in (0, 1)$, and let $\kappa = 1 -\gamma + \frac{3\epsilon}{2}$, and $\alpha = \gamma-1-\epsilon$. Note that $\kappa >0$ and $\alpha <0$, and $\alpha + \kappa = \frac{\epsilon}{2} >0$. Let $c = c_{\Gamma}$ be such that $$\mu(y \in X: d(y, x_0) > R ) < ce^{-R}$$  
\noindent Given $S>0$, let $s_0$ be such that $|s_0| > S$ be such that 
\begin{eqnarray}\label{eq:s0:def} d(h_{s_0} x, x_0) &\geq& \left(\beta -\frac{\epsilon}{2}\right) \log s_0 \nonumber \\ 2cs_0^{-\frac{\epsilon}{2}} &<&  \delta
\end{eqnarray}
\noindent We can pick such an $s_0$ by (\ref{eq:beta:hyp}). Without loss of generality we assume $s_0 >0$. Since we will consider the interval $(-s_0, s_0)$ our arguments are independent of the sign. Let $C = B(x_0, \kappa \log s_0)$. Note that \begin{equation}\label{eq:C:vol} \mu(C) > 1 - cs_0^{-\kappa}\end{equation} and by (\ref{eq:s0:def}), $$d(h_{s_0} x, C) > \left(\beta - \frac{\epsilon}{2} - \kappa\right)\log s_0.$$ Since for all $s \in \R, y \in X$, $d(h_s y, y) \le 2 \log s$, the trajectory $\{h_s x\}_{s \in \R}$ has to have been outside of $C$ for time at least $\tau = t_0^{\gamma-\epsilon}$ since $$2\log \tau = \left(\beta - \frac{\epsilon}{2}-\kappa\right)\log t_0.$$ Thus, \begin{equation}\label{eq:s0:ave} \int_{-s_0}^{s_0} \chi_C(h_s x) ds  \le 2s_0 - s_0^{\gamma - \epsilon},\end{equation} and \begin{equation}\label{eq:C:s0:vol} 2 s_0 \mu(C) > 2s_0 - 2cs_0^{1-\kappa}\end{equation} The second condition in (\ref{eq:s0:def}) yields that $2 s_0 \mu(C) > \int_{-s_0}^{s_0} \chi_C(h_s x) ds$. Thus, we have
\begin{eqnarray}\label{eqnarray:dev:payoff} \nonumber \left| \int_{-s_0}^{s_0} \chi_C( h_s x) ds  - 2s_0 \mu (C)\right| &=& 2s_0 \mu(C) - \int_{-s_0}^{s_0} \chi_C( h_s x) ds \\ \nonumber &\geq& \left(2s_0 - 2cs_0^{1-\kappa} \right) - \left(2s_0 - s_0^{\gamma - \epsilon}\right) \\ \nonumber& =& s_0^{\gamma -\epsilon} - 2cs_0^{1-\kappa} \\ \nonumber & = & s_0^{\gamma - \epsilon} \left( 1 - 2 cs_0^{-\frac{\epsilon}{2} }\right) \\  &>& s_0^{\gamma - \epsilon} \left( 1 - \delta \right)
\end{eqnarray} where the last inequality again follows from the second condition of (\ref{eq:s0:def}). This yields (\ref{eq:hyp:deviations}) and concludes the proof of Theorem~\ref{theorem:hyp:deviations}.\qed

\end{document}